\documentclass[]{article}

\usepackage{makeidx}  

\usepackage{amsthm} 

\usepackage{latexsym}
\usepackage[pdftex]{graphics}
\usepackage{xypic}
\usepackage{amsmath}
\usepackage{amssymb}
\xyoption{all} 
\usepackage{mathrsfs}
\usepackage{eufrak}

\newtheorem{theorem}{Theorem}[section]

\newtheorem{lemma}[theorem]{Lemma}

\newtheorem{definition}[theorem]{Definition}

\newtheorem{proposition}[theorem]{Proposition}

\newtheorem{corol}[theorem]{Corollary}

\newtheorem{remark}[theorem]{Remark}

\newcommand{\C}{{\mathcal C}}

\newcommand{\decode}{\rhd}
\newcommand{\code}{\lhd}

\begin{document}

\title{Girard's $!(\ )$ as a reversible fixed-point operator }
\author{PETER HINES}

\maketitle

\begin{abstract}
We give a categorical description of the treatment of the $!(\ )$ exponential in the Geometry of Interaction system, with particular emphasis on the fact that the GoI interpretation `forgets types'. We demonstrate that it may be thought of as a fixed-point operation for reversible logic \& computation.
\end{abstract}


%
\section{Introduction}
\subsection{Logical background}
The Curry-Howard isomorphism \cite{CH} (also known as the `proofs as programs' correspondence) and later categorical extensions \cite{LS} provide a natural way of looking at {\em logical systems} as {\em computing systems} and vice versa. For applications to reversible and quantum computing, it is common to consider computational interpretation of reversible logics -- in particular, 
J.-Y. Girard's Linear Logic \cite{LL}. This provided a finer-grained decomposition of natural deduction in which the structural operations of {\em copying} and {\em contraction} (i.e. deletion against a copy) were either forbidden, or severely restricted by a strong typing system. As a natural consequence of this, Linear Logic is an essentially reversible logical system, and the computational interpretation provided by extensions of the Curry-Howard isomorphism is in terms of reversible computation.

The computational core of (a significant fragment of) Linear Logic was isolated in the Geometry of Interaction series of papers \cite{GOI1,GOI2,GOI3}. Although the system described was logically degenerate (i.e. conjunction and disjunction were identified, as were universal and existential quantification, and propositions and their negations), the computational core remained --- indeed, as demonstrated in \cite{HA,AHS}, the system described had a natural interpretation as untyped combinatory logic and hence (via the standard embedding of untyped lambda calculus into combinatory logic) was computationally universal, despite being entirely based around reversible primitives.

\subsection{Typing systems in Linear Logic, and untyped computation}
The demonstration of the computational universality of the Geometry of Interaction system given in \cite{AHS} was via an encoding of (untyped) combinatory logic. Indeed, in \cite{GOI1}, Girard states that the system presented `forgets types'. A very natural question is then, {\em `where does this leave the strong typing system used to restrict the structural rules of logic?'} 

Linear Logic itself uses two unary operations, frequently described as modalities,
denoted $!(\ )$ and $?(\ )$, commonly called `of course' and `why not' respectively, or more whimsically `bang' and `whimper'. These are used to provide a typing system that restricts the applicability of potentially irreversible structural rules.

Both these operations survive the passage from Linear Logic to the Geometry of Interaction system (by contrast, it must be noted, with the distinction between conjunction and disjunction, which is lost). This raises the natural question of how an operator, which is supposed to provide a typing system that accommodates irreversible structural operations, can possibly have an interpretation within an entirely reversible, untyped system.

\subsection{The objectives of this paper}
The substantial claims of this paper are the following:
\begin{enumerate}
\item The interpretation of the $!(\ )$ operation (and by duality, the $?(\ )$ operation) within the Geometry of Interaction system is as a fixed-point operation.
\item This fixed-point operation, in stark contrast to fixed-point constructions generally, lives within an entirely reversible setting. 
\item This is possible precisely because the logical/computational interpretation is as an untyped system. 
\end{enumerate}
We take a categorical approach, and -- as is standard -- work within the category of partial injections on sets. We first give explicit description of a monoidal tensor $(\ \star \ )$ that models the (identified) multiplicative conjunction / disjunction of Linear Logic. This is a single-object (i.e. untyped) analogue of the disjoint union within the category of partial injections. 

A `splitting' of this, derived from the quasi-projections and quasi-injections of the category of partial injections, gives rise to an embedding of an algebraic structure known as Girard's dynamical algebra within the same category, and it is shown how iterating this splitting gives rise to a single-object analogue of the Cartesian product at the same object.  

We then describe how the $!(\ )$ operation is defined in terms of this single-object analogue of the Cartesian product, and demonstrate that it is both functorial and satisfies the `fixed-point' condition $f\star !(f)= !(f)$ --- this condition, of course, can only be satisfied within an untyped setting. Finally, the $?(\ )$ operation is shown to be defined in terms of $!(\ )$ and a categorical symmetry isomorphism; and an explicit description of this is also given.

As the above constructions and results rely on a category with two monoidal tensors satisfying a distributivity law, we discuss the similarities and differences between this approach, and an alternative application \cite{PH13a} of the $!(\ )$ operation of linear logic found in quantum computation (and heavily used in Shor's quantum factorisation algorithm) that also relies on categorical distributivity. 

\section{Categorical preliminaries}
\subsection{The category of partial bijections}
The interpretation of the Geometry of Interaction system takes place within the category $\bf pInj$ of partial bijections on sets\footnote{In fact, J.-Y. Girard works within a category of separable Hilbert spaces. However, as observed by many authors \cite{SA96,PH97,HA,AHS}, all the action takes place within the image of $\bf pInj$ under Barr's $l_2:{\bf pInj}\rightarrow {\bf Hilb}$ functor; thus working with $\bf pInj$ itself is sufficient, and the additional linear structure of $\bf Hilb$ plays no essential role.}, defined as follows: 
\begin{definition}{\bf The category of partial bijections}\\
The category $\bf pInj$ of {\bf partial injections} has as objects the proper class of all sets. For all $X,Y\in Ob({\bf pInj})$, the homset ${\bf pInj}(X,Y)$ is the subset of $Y\times X$ satisfying 
\[ y=y' \ \ \Leftrightarrow  \ \ x=x' \ \ \ \ \forall \ (y,x)\ ,\ (y',x')\ \in \ f\subseteq Y\times X \]
Composition is inherited from the category of relations, and it is straightforward that $\bf pInj$ is closed under this composition. It is common to use functional notation, and write $f(x)=y$ for $(y,x)\in f$.
\end{definition}

The category $\bf Pinj$ has a very strong notion of duality:
\begin{definition}{\bf Daggers and generalised inverses}\\
A {\bf dagger} on a category is simply a duality that is the identity on objects -- i.e.  a contravariant endofunctor $( \ )^\dagger:\C\rightarrow \C$ satisfying $ (1_A)^\dagger = 1_A$ and $\left( (f)^\dagger\right)^\dagger =f$, for all $A\in Ob(\C)$ and $f\in \C(A,B)$. 

An {\bf inverse category} is a category $\C$ where every arrow $f\in \C(X,Y)$ has a unique {\bf generalised inverse} $f^\ddag\in \C(Y,X)$ satisfying $ff^\ddag f=f$ and $f^\ddag ff^\ddag =f^\ddag$.
\end{definition}

The generalised inverse of an inverse category is (unsurprisingly) an example of a dagger operation. It is folklore that $\bf pInj$ is an inverse category, with generalised inverse defined by 
\[ f^\ddag=\{ (x,y):(y,x)\in f\}\in {\bf pInj}(Y,X) \ \  \ \ \forall \ f\in {\bf pInj}(X,Y) \]
This operation exhibits the self-duality ${\bf pInj}\cong {\bf pInj^{op}}$.

In contrast to dagger-equipped categories generally, inverse categories have a naturally defined partial order on their hom-sets.
\begin{definition}\label{nat_po}{\bf The natural partial order of an inverse category}\\
Let $\left(\mathcal C,(\ )^\ddag \right)$ be an inverse category. 
 For all $A,B\in Ob({\mathcal C})$, the relation $\unlhd_{A,B}$ is defined on ${\mathcal C}(A,B)$, as follows:
\[ f\unlhd_{A,B} g \ \mbox{ iff } \ \exists \ e^2=e\in {\mathcal C}(A,A) \ s.t. \ f=ge \]
For all $A,B\in Ob({\mathcal C})$, the relation $\unlhd_{A,B}$ is a partial order on ${\mathcal C}(A,B)$, called the {\bf natural partial order}.
When it is clear from the context, it is standard to omit the subscript on $\unlhd$.
\end{definition}
The category ${\bf pInj}$ is the `prototypical' inverse category, in that all inverse categories arise as subcategories of $\bf pInj$ (See \cite{CH} for this result, and \cite{RC} for similar in small categories. These results are based on the classic Wagner-Preston representation theorem for inverse semigroups \cite{GP,VW}, which itself generalises Cayley's result for groups).

\subsection{Monoidal tensors on ${\bf pInj}$}
The category $\bf pInj$ has two monoidal tensors, the Cartesian Product $\_ \times \_$, and the Disjoint Union $\_ \uplus \_$. Neither of these are products (\& hence, by duality, nor coproducts). Further, $\bf pInj$ is not closed; however, it does have a categorical trace \cite{JSV} with respect to disjoint union, and this trace provides the dynamics of the cut-elimination process \cite{SA96,AHS,PH97}. Thus, we may consider the dynamics of the Geometry of Interaction system to be give by composition within the compact closed category $\bf Int(pInj)$. This is equivalent to the  $\bf GoI(pInj)$ of \cite{SA96}; however, we wish to distinguish between Girard's GoI {\em system}, and Abramsky's GoI {\em categorical construction}, so use Joyal, Street, \& Verity's terminology throughout.

\section{GoI connectives as untyped monoidal tensors}
The Geometry of Interaction system (at least, the first two parts \cite{GOI1,GOI2}) is based on the multiplicative fragment of Linear Logic \cite{LL} only\footnote{The extension to the additives given in \cite{GOI3} is of very different character, requiring significantly different categorical models, and beyond the scope of this paper.}, essentially without quantifiers and units. Thus, in a categorical setting, we may work with monoidal categories that are not assumed to have a unit object (see definition \ref{semicat} below).Two further aspects of the GoI system are particularly relevant:
\begin{enumerate}
\item As stated in \cite{GOI1}, the system {\em ``forgets types''}.  Categorically, as with the untyped lambda calculus, the interpretation takes place within a single-object category, or monoid.
\item The connectives {\em tensor} and {\em par} are identified. Thus the categorical interpretation requires a single connective (i.e. semi-monoidal tensor) and corresponding adjoint notion of closure (internal hom, or trace). 
\end{enumerate}
The requirement that the GoI system `forgets types' means that -- taking the standard {\em `types as objects within a category'} interpretation -- the action takes place within a single-object category. For a single-object {\em monoidal} category to be non-trivial, the unique object cannot be a unit object for the tensor; this is axiomatised within the theory of {\em semi-monoidal categories} below.

\subsection{Single-object and unitless monoidal categories}
The following straightforward definition, taken from \cite{PH13d}, axiomatises categories that have all the structure of a monoidal category except perhaps for a unit object: 
\begin{definition}\label{semicat}
Let $\C$ be a category. We say that $\C$ is {\bf semi-monoidal} when there exists a functor 
$( \_\otimes \_ ) : \C\times \C \rightarrow \C$ that we call the {\bf tensor},
together with a natural object-indexed family of {\bf associativity isomorphisms}
\[ \{  \  \tau_{A,B,C} : A\otimes (B\otimes C)\rightarrow (A\otimes B) \otimes C \} _{A,B,C\in Ob(\C)} \]
satisfying MacLane's {\bf pentagon condition} 
\[ ( \tau_{A,B,C}\otimes 1_D)   \tau_{A,B\otimes C,D}  (1_A \otimes  \  \tau_{B,C,D}) =    \tau_{A\otimes B,C,D}  \tau_{A,B,C\otimes D} \]

When there also exists a natural object-indexed natural family of {\bf symmetry isomorphisms} $\{ \sigma_{X,Y}:X\otimes Y\rightarrow Y\otimes X \}_{X,Y\in Ob(\C) }$
satisfying MacLane's {\bf hexagon condition} $\tau_{A,B,C}\sigma_{A\otimes B,C}  \tau_{A,B,C} =(\sigma_{A,C}\otimes1_B)  \tau_{A,C,B} (1_A\otimes \sigma_{B,C})$ 
we say that $(\C,\otimes ,  \tau,\sigma)$ is a {\bf symmetric semi-monoidal category}. 
A semi-monoidal category $(\C,\otimes, \  \tau_{\_,\_,\_})$ is called {\bf strictly associative} when $  \  \tau_{A,B,C}$ is an identity arrow\footnote{This is not implied by equality of objects $A\otimes (B\otimes C) = (A\otimes B)\otimes C$, for all $A,B,C\in Ob(\C)$, even when this equality is well-defined (e.g. within small categories). Although MacLane's pentagon condition is trivially satisfied by the appropriate identity arrows, naturality with respect to the tensor may nevertheless fail. The examples we present in Section \ref{untyped_connective} onwards  illustrate this phenomenon.}, for all $A,B,C\in Ob(\C)$. 
A functor $\Gamma : \C \rightarrow \mathcal D$ between two semi-monoidal categories $(\C ,\otimes_\C )$ and $(\mathcal D,\otimes_\mathcal D)$ is called (strictly) {\bf semi-monoidal} when $\Gamma (f\otimes_\mathcal C g) = \Gamma(f) \otimes_\mathcal D \Gamma(g)$.
A semi-monoidal category $(\C,\otimes)$ is called {\bf monoidal} when there exists a {\bf unit object} $I\in Ob(\C)$, together with, for all objects $A\in Ob(\C)$, distinguished isomorphisms 
$\lambda_A:I \otimes A \rightarrow A$ and $\rho_A : A\otimes I \rightarrow A$
satisfying MacLane's {\bf triangle condition} $1_U\otimes \lambda_V = (\rho_U\otimes 1_V) \tau_{U,I,V}$ for all $U,V\in Ob(\C)$.
\end{definition}

Clearly, any single-object monoidal category is trivial; the unique object must be the unit object, with all that entails. However, single-object semi-monoidal categories may have a much richer theory.

\begin{definition} we define an {\bf untyped monoidal category} to be a single-object semi-monoidal category.
\end{definition}

\begin{remark}{\em Unitless monoidal categories, and categorical coherence} A natural question is whether MacLane's coherence theorems are still applicable in the unitless or semi-monoidal setting? Based on the theory of Saavedra units, an appendix to \cite{PH13d} gives a method of adjoining a (strict) unit object to a semi-monoidal category that is right-inverse to the obvious forgetful functor. Thus, all the standard theories of coherence for associativity, symmetry, distributivity, \&c. remain applicable.
\end{remark}

\subsection{The single untyped connective of GoI}\label{untyped_connective}
Up to the embedding $l_2:({\bf pInj},\uplus)\rightarrow ({\bf Hilb},\oplus)$, the constructions of the Geometry of Interaction system (parts I, II) take place within the endomorphism monoid of a single countably infinite set --- for simplicity we will take this to be the natural numbers $\mathbb N$. The significant feature of $\mathbb N$ that allows its use in this setting is the fact that it is {\em self-similar}.
\begin{definition}\label{basicselfsim}
Let $(\C,\otimes)$ be a semi-monoidal category. An object  $S\in Ob (\mathcal C)$ is called {\bf self-similar} when it satisfies $S\cong S\otimes S$. 
Making the arrows exhibiting this self-similarity explicit, we define a 
{\bf self-similar structure} $(S,\code,\decode )$ to be an object 
$S\in Ob(\C)$, together with two mutually inverse arrows
\begin{itemize}
\item {\bf (code)} $\code\in \C(S\otimes S, S)$.
\item {\bf (decode)}  $\decode \in \C(S,S\otimes S)$.
\end{itemize}
satisfying $\decode\code = 1_{S\otimes S}$ and $\code \decode= 1_S$.

It is proved in \cite{PH13d} that self-similar structures are unique up to unique isomorphism; however, actual uniqueness forces a collapse to the unit object.
\end{definition}
It is straightforward from simple `Hilbert hotel' style reasoning that $\mathbb N$ is self-similar with respect to both disjoint union and Cartesian product, so 
\[ \mathbb N \times \mathbb N \ \cong \ \mathbb N \ \cong \ \mathbb N \uplus \ \mathbb N \]

The particular self-similar structure that Girard uses to exhibit the self-similarity $\mathbb N \cong \mathbb N \uplus \mathbb N$ is given by the {\em Cantor pairing}:
\begin{definition}\label{cantor_pairing}
Using the explicit description of the disjoint union as  $\mathbb N \uplus \mathbb N = \mathbb N \times \{ 0,1\}$, the {\bf Cantor pairing} $\code: \mathbb N \uplus \mathbb N \rightarrow \mathbb N $ is the bijection
$\code(n,i) = 2n+i$. 

Its (global) inverse $\decode:\mathbb N \rightarrow \mathbb N \uplus \mathbb N$ is given by 
\[ \decode(n) = \left\{ \begin{array}{lr}  \left(\frac{n}{2},0\right) & \mbox{$n$ even,} \\
			&			\\
							\left(\frac{n-1}{2},1\right) & \mbox{$n$ odd.} \\
							\end{array}\right.
							\]
\end{definition}
(We will demonstrate in Section \ref{internal_times} how a self-similar structure exhibiting the self-similarity $\mathbb N \cong \mathbb N \times \mathbb N$ may be derived from the above Cantor pairing).

The single connective of \cite{GOI1,GOI2} is then modelled by the following operation:
\begin{definition}\label{star_def}
We define
\[ \_ \star\_ : {\bf pInj}(N,N)\times {\bf pInj}(N,N) \rightarrow {\bf pInj}(N,N) \]
by, for all $f,g: \mathbb N \rightarrow \mathbb N$, 
\[ f\star g \ = \ \code (f\uplus g ) \decode \]
Diagramatically, 
\[ 
\xymatrix{ 
\mathbb N \uplus \mathbb N \ar[rr]^{f\uplus g} &		&	\mathbb N \uplus \mathbb N \ar[d]^{\code}  \\
\mathbb N \ar[u]^\decode \ar[rr]_{f\star g} 		&		& \mathbb N 
}
\]
\end{definition}
As shown in \cite{PH97,PH99}, this gives $({\bf pInj} (\mathbb N ,\mathbb N),\_ \star \_)$ all the structure of a symmetric monoidal category apart from the unit object; giving what \cite{PH13d} refers to as a {\em unitless monoidal category}. We refer to \cite{PH13b} for coherence results relating associativity, self-similarity, and untypedness,  and \cite{PH13c} for elementary arithmetic proofs that the associativity and symmetry isomorphisms 
\[  \tau (n) =\left\{ \begin{array}{lr}
2n &  \ \ \ n\ (mod \ 2)=0, \\
		& 	\\
n+1 &  \ \ \ n\ (mod \ 4)=1,  \\	
		& 	\\
\frac{n-1}{2} &  \ \ \  n\ (mod \ 4)=3.  \\	
\end{array}\right. \ \ , \ \ \sigma (n) =\left\{ \begin{array}{lr} n+1 & \ \ \ n \mbox{ even,} \\ 	&	\\ n-1 & \ \ \ n \mbox{ odd.} 
\end{array}\right.
\]
satisfy MacLane's Pentagon and Hexagon conditions \cite{MCL}.

\subsection{The Geometry of Interaction and models of $\lambda$-calculus}
So far, the categorical interpretation of Girard's GoI system has lead to a single-object semi-monoidal closed category; there is an obvious comparison to be made with the $C$-monoids (single-object Cartesian closed categories without unit objects) modelling the pure untyped $\lambda$- calculus. Significant differences are that the single-object tensor in Girard's system is neither a product nor a coproduct, and the form of categorical closure is {\em compact closure} rather than {\em Cartesian closure}. The absence of the universal property associated with categorical products translates into the failure of copying (and, by the dualities of $\bf pInj$, the failure of contraction). 

\section{The untyped $!(\ )$}
The computational power of GoI is recovered by the use of the $!(\ )$ operation, giving (as shown in \cite{AHS}) the computationally universal linear combinatorial algebra.
We now give a categorical description of how Girard modelled the $!(\ )$ operation {\em in this untyped setting}, together with the interaction between the $!(\ )$ and the single untyped tensor used to model the connectives. 

Let us denote the endomorphism monoid ${\bf pInj}(\mathbb N ,\mathbb N)$ by $End(\mathbb N)$. Recall the untyped monoidal tensor 
\[ \_\star\_  : End(\mathbb N)\times End(\mathbb N) \rightarrow End(\mathbb N) \]
defined in terms of the Cantor pairing (Definition \ref{cantor_pairing}) by 
\[ f\star g\ = \ \code (f\uplus g)\decode \ \ \ \forall \ f,g\in End(\mathbb N) \]
The key to the untyped version of the $!(\ )$ is `splitting' the Cantor pairing (and hence the untyped tensor) into a construction based on two partial injections.

\subsection{Splitting the untyped tensor}\label{split}
Although disjoint union $\_ \uplus \_ : {\bf pInj}\times {\bf pInj}\rightarrow {\bf pInj}$ is neither a product nor a coproduct is nevertheless has `quasi-projections' 
\[ \xymatrix{ 
			&	A \uplus B \ar[dl]_{\pi_0} \ar[dr]^{\pi_1}	& 		\\
A			&									&	B 
}
\]
defined in the obvious way by 
\[ \pi_0=\{ ((a,0),a) : a\in A\} \ \mbox{ and } \ \pi_1=\{ ((b,1),b) : b\in B\} .\]
Taking generalised inverses gives the `quasi-injections'
\[ 
\xymatrix{
				& A \uplus B	&			\\
A \ar[ur]^{\iota_0}	&			& B \ar[ul]_{\iota_1} 
}
\]
Together, these satisfy the relations
\[ \begin{array}{ccc}
\pi_0\iota_0=1_A	&	\ \ \ \ 	&	\pi_1\iota_1=1_B \\
				&		&		\\
\pi_1\iota_0=0_{AB}	&	&	\pi_1\iota_2=0_{BA}
\end{array}
\]
We may again use the Cantor pairing to construct untyped analogues of these quasi- projections/injections at the endomorphism monoid of $\mathbb N$.\footnote{This process of constructing `untyped analogues' of categorical structures is of course functorial \cite{PH97,PH99,PH13d}, and may be used to construct untyped versions of many categorical properties, including the trace and compact closure used to model the dynamics of cut-elimination \cite{PH97,PH99}. It is properly thought of as a `strictification' procedure applied to the self-similarity exhibited by the Cantor pairing, in an analogous manner to MacLane's strictification procedure for associativity. See  \cite{PH13d} for details.}
\begin{definition}\label{dynamical}
We define the {\bf untyped quasi-projections} $p,q\in End(\mathbb N)$ as follows:
\[ \xymatrix{
\mathbb N \uplus \mathbb N \ar[rr]^{\pi_0} &	& \mathbb N 	& 	&	\mathbb N \uplus \mathbb N \ar[ll]_{\pi_1}	\\
\mathbb N \ar[u]^\decode \ar[urr]_p		&	&			&	&  \mathbb N \ar[ull]^q \ar[u]_\decode 
}
\]	
Taking generalised inverses of the above diagram gives the {\bf untyped quasi-injections} $p^\ddag,q^\ddag\in End(\mathbb N)$, as follows:
\[ \xymatrix{
\mathbb N \uplus \mathbb N \ar@{<-}[rr]^{\iota_0} &	& \mathbb N 	& 	&	\mathbb N \uplus \mathbb N \ar@{<-}[ll]_{\iota_1}	\\
\mathbb N \ar@{<-}[u]^\code \ar@{<-}[urr]_{p^\ddag}		&	&			&	&  \mathbb N \ar@{<-}[ull]^{q^\ddag} \ar@{<-}[u]_\code 
}
\]	
 \end{definition}
 We may give explicit descriptions of the above partial injections, as follows:
 \begin{lemma}\label{simplesums}
 Let $p,q,p^\ddag,q^\ddag\in End(\mathbb N)$ be the untyped quasi- projections/injections defined in terms of the Cantor pairing, as above. Then $p,q\in End(\mathbb N)$ are the following bijections:
 \[ p(n) = 2n \ \mbox{ and } q(n)=2n+1 \]
 and similarly $p^\ddag,q^\ddag\in End(\mathbb N)$ are the following partial bijections:
 \[ p^\ddag(n) = \left\{ \begin{array}{lr} \frac{n}{2} & \ \ \ \ n \mbox{ even,} \\ & \\ \mbox{undefined} & \ \ \ \ \mbox{otherwise.} \end{array} \right.
 \ \ \mbox{ and } 
 q^\ddag(n) = \left\{ \begin{array}{lr} \frac{n-1}{2} & \ \ \ \ n \mbox{ odd,} \\ & \\ \mbox{undefined} & \ \ \ \ \mbox{otherwise.} \end{array} \right.
 \]
 \end{lemma}
 \begin{proof}
 This follows immediately by expanding out the above definitions.
  $\Box$ \end{proof}
 The above explicit form makes it apparent how $p,q$ (resp. $p^\ddag,q^\ddag$) may be thought of as a `splitting' of the Cantor pairing (resp. its inverse). 
 The untyped quasi- projections/injections give rise to Girard's {\em dynamical algebra}, as we now demonstrate:
 
 \begin{proposition}
 Let $p,q,p^\ddag,q^\ddag\in End(\mathbb N)$ be as defined above. Then 
 \[ pp^\ddag=1_\mathbb N =qq^\ddag \ \ \mbox{ and } \ \ pq^\ddag = 0_\mathbb N = qp^\ddag \]
 Further, when we consider the partial order on $End(\mathbb N)$ provided by the inverse category structure (Definition \ref{nat_po}),
 \[ p^\ddag p\vee q^\ddag q = 1_\mathbb N \] 
 \end{proposition}
 \begin{proof}
 These are all immediate from the explicit description of the untyped quasi- projections/injections given in Lemma \ref{simplesums} above.
  $\Box$ \end{proof}
 
 The untyped tensor, modelling the sole connective of the GoI system, may also be given in these terms:
 \begin{corol}
 Let $\_ \star \_ : End(\mathbb N ) \times End(\mathbb N)\rightarrow End(\mathbb N)$ be as in Definition \ref{star_def}. Then for all $f,g\in End(\mathbb N)$, the join (with respect to the natural partial order $\unlhd$ induced by the generalised inverse) $p^\ddag fp \vee q^\ddag f q$ exists, and satisfies
 \[ f\star g \ = \ p^\ddag fp \vee q^\ddag gq \]
 \end{corol}
 \begin{proof}
 This again simply follows by expanding out the basic definitions.
  $\Box$ \end{proof}
 
 \subsection{Untyped analogues of the Cartesian product}\label{internal_times}
 The splitting of the Cantor pairing (and hence the untyped tensor) given in Section \ref{split} above may be used to construct an isomorphism $\mathbb N \times \mathbb N \cong \mathbb N$ in terms of the Cantor pairing exhibiting the isomorphism $\mathbb N \uplus \mathbb N \cong  \mathbb N$. 
 
 \begin{definition}\label{psidef}
 we define the {\bf exponential bijection} to be the isomorphism $\psi:\mathbb N \times \mathbb N \rightarrow \mathbb N$ given by 
 \[ \psi(a,b) \ = \ q^b p(a) \]
 where $p,q\in End(\mathbb N)$ are the untyped quasi-projections given in Definition \ref{dynamical}. Using the explicit description of $p,q$ from Lemma \ref{simplesums} gives the following formula: 
 \[ \psi(x,y) \ = \ 2^{y+1}x+ 2^{y}-1 \]
 \end{definition}
 
 \begin{remark}
 The construction of a self-similar structure $(\mathbb N,\psi, \psi^{-1})$ of $({\bf pInj},\times)$  from a self-similar structure $(\mathbb N ,\code,\decode)$ of $({\bf pInj},\uplus )$ given above 
clearly relies on the fact that $\mathbb N$ is countable.  Although we may perform similar constructions with uncountable self-similar objects of $\bf pInj$ -- such as the Cantor set $\mathfrak C$ -- this does not in general result in isomorphisms exhibiting the self-similarity $\mathfrak C \times \mathfrak C\cong \mathfrak C$; rather, we get a bijection $\mathfrak C \times \mathbb N \ \cong \ \mathfrak C$.

Even when working with countably infinite objects such as  $\mathbb N \in Ob({\bf pInj})$, we need to ensure that any self-similar structure $(\mathbb N , \phi :\mathbb N \uplus \mathbb N \rightarrow \mathbb N ,\phi^{-1}: \mathbb N \rightarrow \mathbb N \uplus \mathbb N)$ used satisfies the `no-residue' condition 
 \[ \bigcap _{j=0}^\infty \phi\left( \mathbb N , 1\right) \ = \ \emptyset \]
 This condition is, of course, satisfied by the Cantor pairing.
 \end{remark}
 
The untyped tensor $\_ \star \_$ of Definition \ref{star_def} may be thought of as a single-object analogue of the disjoint union; in a similar way,  we may use the bijection $\psi:\mathbb N \times \mathbb N \rightarrow \mathbb N$ to form an untyped analogue of the Cartesian product:
 
 \begin{definition}
 We define the {\bf exponential tensor} $\_ \odot\_ : End(\mathbb N )\times End(\mathbb N)\rightarrow End(\mathbb N)$ in terms of the exponential bijection of Definition \ref{psidef} above, as follows:
 \[ f \odot g \ = \ \psi (f\times g) \psi^{-1} \]
 Diagramatically, 
 \[ \xymatrix{ 
 \mathbb N \times \mathbb N \ar[rr]^{f \times g} 	&		&  \mathbb N \times \mathbb N \ar[d]^\psi \\
 \mathbb N \ar[u]^{\psi^{-1}} \ar[rr]_{f\odot g}	&		& \mathbb N 
 }
 \]
 \end{definition}
 This is again a unitless monoidal tensor on $End(\mathbb N)$ satisfying associativity and symmetry up to canonical isomorphisms (see \cite{PH13d} for a general construction, of which this is a special example, and \cite{PH97} for an explicit description of the associativity and symmetry isomorphisms.
 
 \subsection{Constructing the $!(\ )$}
 The operation used in \cite{GOI1,GOI2} to model the $!(\ )$ operation may be defined in terms of the above `exponential tensor' as follows:
 \begin{definition}
 We define the {\bf bang} operation $!(\ ) : End(\mathbb N)\rightarrow End(\mathbb N)$ by 
 \[ !(f) \ = \ (1_\mathbb N \odot f) \]
 \end{definition}
 
 Basic properties of $!(\ ):End(\mathbb N)\rightarrow End( \mathbb N)$ follow by functoriality.
 \begin{lemma}
 The operation $!(\ ):End(\mathbb N ) \rightarrow End(\mathbb N)$ defined above is a monoid homomorphism.
 \end{lemma}
 \begin{proof}
 By the functoriality of  $\_ \odot\_ : End(\mathbb N )\times End(\mathbb N)\rightarrow End(\mathbb N)$,
 \[ !(1_\mathbb N) = 1_\mathbb N \ \ \mbox{ and } \ \ !(g)!(f)= !(gf) \]
 Thus $!(\ )$ is a monoid homomorphism.
  $\Box$ \end{proof}
 
 From the explicit description of $\_ \odot\_ : End(\mathbb N )\times End(\mathbb N)\rightarrow End(\mathbb N)$, and hence of $!(\ ):End(\mathbb N)\rightarrow End(\mathbb N)$ in terms of the Cantor pairing, we may write $!(f)$ explicitly in terms of the untyped quasi-projections (i.e. the generators of the dynamical algebra) as follows:
 \begin{proposition}
 Given arbitrary $f\in End(\mathbb N)$, then $!(f)\in End(\mathbb N)$ may be given explicitly as an infinite join in the natural partial order, as follows:
 \[ !(f)\ = \ p^\ddag f p \ \vee\ q^\ddag p^\ddag f pq \ \vee \ (q^\ddag)^2 p^\ddag f p q^2 \ \vee \ (q^\ddag)^3 p^\ddag f p q^3 \ \vee \ \ldots \]
 \end{proposition}
 \begin{proof}
 Again, this is simply by expanding out the definition.
  $\Box$ \end{proof}
 
 As a corollary of the above explicit description, the $!(\ )$ operation provides a  natural, functorial `fixed-point' for the untyped tensor:
 \begin{corol}
 The interaction between the {\em bang} operation, and the untyped tensor is the following:
 \[ f \star !(f) \ = \ !(f) \]
 \end{corol}
 \begin{proof}
 From the explicit description of the untyped tensor in terms of the untyped quasi- projections/injections, 
 \[ f\star !(f) \ = \ p^\ddag f p \vee q^\ddag !(f) q \]
 Combining this with the above explicit description of the bang, and using the distributivity of composition over join in the natural partial order of $\bf pInj$ gives
\[ \begin{array}{rl} f \star !(f) 	& = p^\ddag f p \vee q^\ddag !(f) q  \\
						&																																	\\
						& = p^\ddag f p \vee q^\ddag \left(   \ p^\ddag f p \ \vee\ q^\ddag p^\ddag f pq \ \vee \ (q^\ddag)^2 p^\ddag f p q^2 \ \vee \ (q^\ddag)^3 p^\ddag f p q^3 \ \vee \ \ldots \right) q \\
						&																																	\\
						& = \ p^\ddag f p \ \vee\ q^\ddag p^\ddag f pq \ \vee \ (q^\ddag)^2 p^\ddag f p q^2 \ \vee \ (q^\ddag)^3 p^\ddag f p q^3 \ \vee \ \ldots 
 	\end{array}
\]
  $\Box$ \end{proof}
 
 \begin{remark}{\em The dual connective to the bang} \\
 The dual connective to the bang, $?(\ ):End(\mathbb N) \rightarrow End(\mathbb N)$ is modelled similarly; Girard simply uses 
 \[ ?(g) \ = \ (g \odot 1_\mathbb N ) \]
 The properties of this may be derived from the properties of the bang $!(\ )$ and the canonical symmetry map $\sigma \in End(\mathbb N)$ for the untyped analogue of the Cartesian product, by the identity 
 \[ ?(g) \ = \ \sigma !(g ) \sigma \]
 However, it is interesting also to give an explicit description of the $?(\ )$ operation in terms of the generators of the dynamical algebra. Let us define a countably infinite set of arrows $\{ r_j \in End(\mathbb N) \}_{j\in \mathbb N}$ by 
 \[ r_j \ = \ q^j p \ \ \forall j\in \mathbb N \]
 These arrows satisfy the following relations:
 \[ r_jr_k^\ddag = \left\{ \begin{array}{lr} 1_\mathbb N & \ \ j=k \\ 0_\mathbb N & \ \ j\neq k \end{array}\right. \ \ \ \mbox{ and } \ \ \ \bigvee_{n=0}^\infty r_n^\ddag r_n \ = \ 1_\mathbb N \]
 giving an infinitary analogue of the generating relations of Girard's dynamical algebra.
 By construction, for arbitrary $g\in End(\mathbb N)$, the dual connective to the bang may be given an explicit form as:
 $?(g) = \bigvee_{n\in dom(g)} r^\ddag_{g(n)} r_n$. If we adopt the notational convention that $r^\ddag_{g(n)} r_n=0_\mathbb N$ when $g(n)$ is undefined, we may simply write 
 \[ ?(g) = \bigvee_{n=0}^\infty r^\ddag_{g(n)} r_n \]
 \end{remark}
 
 \section{The exponentials of GoI, and categorical distributivity}
 The constructions in this paper live within the category $\bf pInj$, and rely on the fact that a self-similar structure $(\mathbb N ,\code ,\decode)$ in $({\bf pInj},\uplus)$ may be `split and iterated' to produce a self-similar structure  $(\mathbb N ,\psi ,\psi^{-1})$ in $({\bf pInj},\times)$. The untyped analogue of the Cartesian product $(\_ \odot \_ )$ then provides a `fixed-point' operation for $(\ \star \ )$, the untyped analogue of the disjoint union. 
 
 A key point is that the typing of this fixed-point condition, $!(f) = f \star!(f)$, forces the semi-monoidal tensor $\_ \star \_$ to live within an untyped setting. 
 
 A very natural question is then: what is the relationship of this construction with the categorical theory of distributivity, and in particular, the approach to finitary analogues of $!(\ )$ given in \cite{PH13a}? The category $({\bf pInj},\times , \uplus)$ is a {\em distributive category}\footnote{Note: not a {\em linear distributive category}, in the sense of \cite{BCST}.} in the sense of \cite{ML1,ML2}, and (up to canonical isomorphisms satisfying the required coherence isomorphisms) satisfies 
 \[ A\times (B \uplus C ) \ \ \cong \ \ A \times B \ \uplus \ A \times C \]
 It is, of course, possible to use similar techniques to those of \cite{PH13d} to construct 
a single-object semi-monoidal category with two tensors satisfying untyped analogues of Laplaza's coherence conditions for distributivity. 

However, it seems that the actual constructions of $!(\ )$ and $?(\ )$ described above do not rely on this categorical distributivity in an essential way. Certainly, none of the constructions or proofs seems to require distributivity. 
\\

This is in sharp contrast with the system presented in \cite{PH13a}, where -- based on Shor's quantum algorithm --  the unit objects in a distributive category are used to define an operation 
\[ !^n(U) \ = \ 1 \oplus U \oplus U^2 \oplus \ldots \oplus U^{2^n-1} \]
on unitary maps in finite-dimensional Hilbert space (in fact, within the quantum circuit paradigm). This differs from the constructions described above in a number of ways; the most significant is that it relies on the unit objects in an essential way, via an iterated analogue of the fan-out operation \cite{FANOUT}, whereas the constructions based on the Geometry of Interaction necessarily take place within an untyped (i.e. single-object) category, where we cannot have unit objects for either tensor.  

\section{Conclusions} 
It is frequently assumed that the computational content of reversible logics must be severely restricted; it is equally common to claim that fixed-point operations must introduce some essential irreversibility into computing systems. It is interesting to observe how the Geometry of Interaction provides a counterexample to both these inaccurate assumptions. 

Of equal interest is the fact that this reversible fixed-point operation necessarily lives within an untyped system. As well as raising interesting and deep questions about categorical models (see \cite{PH13d} for implications for MacLane's theory of coherence for associativity and \cite{PH13c} for unexpected concrete applications to modular arithmetic), this raises practical questions about the sort of type systems that may be imposed on any reversible programming language that uses a fixed-point operation as a computational primitive.

 \bibliographystyle{plain}
\bibliography{bang_refs}

\end{document}